\documentclass[12pt]{amsart}

\usepackage{amssymb,amsmath,psfrag}
\usepackage{graphics,graphicx}

\usepackage[usenames]{color}
\usepackage{graphicx}
\usepackage{amscd}
\usepackage[colorlinks=true,
linkcolor=webgreen,
filecolor=webbrown,
citecolor=webgreen]{hyperref}

\definecolor{webgreen}{rgb}{0,.5,0}
\definecolor{webbrown}{rgb}{.6,0,0}
\usepackage{amsthm}
\usepackage{fullpage}
\usepackage{enumerate}
\usepackage{epstopdf}
\usepackage{graphics,amsmath,amssymb}

\usepackage{amsfonts}
\usepackage{latexsym}
\usepackage{url}

\usepackage{amssymb,amsmath,psfrag}
\usepackage{graphicx}

\theoremstyle{plain}
\newtheorem{theorem}{Theorem}

\newtheorem{lemma}[theorem]{Lemma}
\newtheorem{proposition}[theorem]{Proposition}

\theoremstyle{remark}

\begin{document}

\begin{center}

\Large{\bf Honeycombs in the Pascal triangle and beyond}
\normalsize

{\sc Matthew Blair, Rigoberto Fl\'orez, and Antara Mukherjee} \footnote {Several of the main results in this paper were found by the first author while working on their
undergraduate research project under the guidance of the second and third authors (who followed the guidelines given  in \cite{FlorezMukherjeeED}).}

{\sc The Citadel }

\footnotesize

 {\tt blairm@citadel.edu, rigo.florez@citadel.edu, antara.mukherjee@citadel.edu}
\normalsize


\end{center}


\footnotesize {\it Abstract:} {
In this paper we present a geometric approach to discovering some known and some new identities using triangular arrays. Our main aim is to demonstrate how to use the geometric patterns (by Carlitz), in the Pascal and Hosoya triangles to 
rediscover some classical identities and integer sequences. Therefore, we use new techniques in classical settings which then provide a new perspective in undergraduate research. }


\section {Introduction}
This paper was partially published in  \emph{Math Horizons}, February 2022, (see \cite{Blair1}). In this version we include the proofs of the results given in \cite{Blair1}. 

The idea of this paper is to engage students ---with no experience in undergraduate research--- in a research project where they can explore geometric patterns in different triangular arrays. In many cases, as we see in this paper, the experiments the students' conduct lead to discovery of well-known combinatorial or Fibonacci identities and in certain rare occasions a new identity.

The experiments presented in this paper are based on a honeycomb pattern in the Pascal and Hosoya triangles where the hexagonal tiles are constructed  
using the entries of the triangles as vertices, (see Figure \ref{alternatingpatterns:2}). The honeycomb pattern embedded in the Hosoya triangle and the  
patterns in Figure \ref{carlitz_patterns:1} (called Carlitz's patterns, \cite{Blair1}) have been explored by Blair et al. already and recently the results obtained from these experiments were  
published recently  in the article titled ``Honeycombs in the Hosoya Triangle", \cite{Blair1}.
The techniques presented here helps the students use geometry to learn about combinatorics and several integer sequences  present in OEIS \cite{sloane}.
 The aim of this paper is to help the students acquire the skills in undergraduate research discovering new and known identities using geometry instead  
 of classic combinatorics. The triangles and the geometric patterns have proved to be an effective tool in this research which led to the discovery of some 
 new identities, rediscovery of some well-known identities and integer sequences present in OEIS \cite{sloane}. Therefore, new techniques used in classical settings can often provide a new perspective in undergraduate research.

\begin{figure}[htbp]
\begin{center}
\includegraphics[width=16cm]{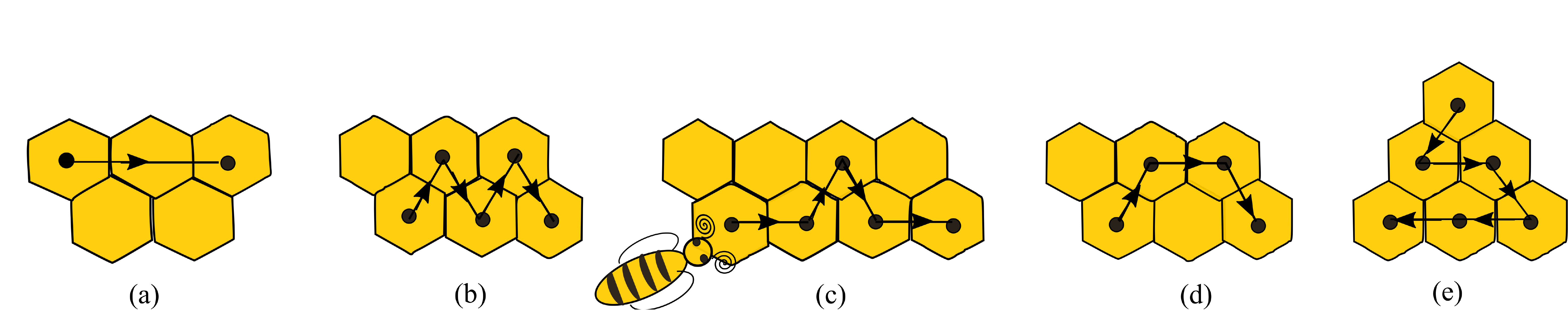}
\caption{Carlitz's patterns.}
\label{carlitz_patterns:1}
\end{center}
\end{figure}

 Note that the honeycomb pattern seen in Figure \ref{alternatingpatterns:2}, is different from the one normally associated with the Pascal triangle.
Furthermore, using the cells and the movement of the bees through these cells as given by Carlitz (see \cite{koshy} or Figure \ref{carlitz_patterns:1}), we obtain different interpretations of some well-known binomial identities associated with the Pascal 
and Fibonacci identities associated with the Hosoya triangle. One of the properties that was previously obtained by using these hexagonal tiles in both the Pascal and the Hosoya triangle is called
the Gould's identity or Star of David theorem (see \cite{florezjunes} and \cite{gould}).  In case of the Pascal triangle, the identity is obtained by showing the greatest common divisor ($\gcd$) of three alternating vertices of a hexagonal tile is equal to the $\gcd$ of the remaining three alternating vertices. On the other hand,
 in the Hosoya triangle a parallel identity is obtained by considering six (hexagonal) tiles forming a hexagonal pattern and showing that the $\gcd$ of the entries present inside three alternating tiles is the same as that of the three remaining hexagonal tiles in the pattern.

In Section \ref{other:triangle} we also provide examples of three other triangular arrays, namely the Lucas triangle, Josef's triangle, and Leibniz's harmonic triangle. Some of the patterns seen in Figure \ref{carlitz_patterns:1} 
can be explored in these triangles as well therefore providing undergraduate students new directions in research.

\section{Carlitz's patterns in the Pascal triangle}\label{pascal:triangle}

In this section we use Carlitz's patterns seen in Figure \ref{carlitz_patterns:1}  to rediscover some well-known binomial identities. The patterns in Figure \ref{carlitz_patterns:1} (a)--(d) can also be seen as Patterns 1, 2, in Figure \ref{hex_patterns:1}(a) and Patterns 3, 4 in Figure \ref{hex_patterns:1}(b).
Note that the identity present in Proposition \ref{propo:3} can only be found partially in literature while the identity in Theorem \ref{theorem:4} is new.

 Here we observe that the hexagonal tiles used in the Pascal triangle is constructed using the entires of the triangle in every other row as shown in Figure \ref{hex_patterns:1}. In this figure, each dot is an entry of the triangle.
This setup of the hexagonal tiles can also be seen in Figure \ref{alternatingpatterns:2}.
This construction of the honeycomb pattern can also be seen in Figure \ref{alternatingpatterns:2}. 

Based on this honeycomb pattern we can construct another triangular array with entries of the from each hexagonal cells. These entries are of the form $\binom{2k}{i}$ where $k>0$ and $i$ is odd.
The entries of this triangular array are $ 2,  4, 4, 6, 20, 6, 8, 56, 8,\ldots$. We leave it to the readers to conduct 
\begin{figure}[htbp]
\includegraphics[scale=0.30]{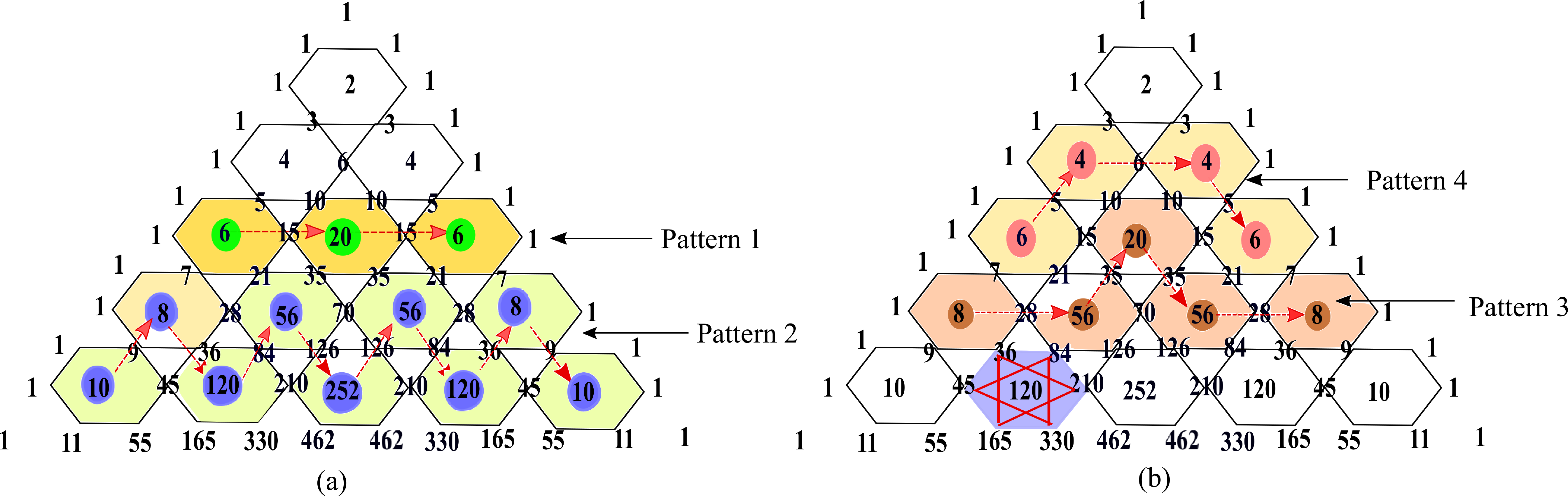} 
\caption{Hexagonal cells in the triangles and patterns.}
\label{hex_patterns:1}
\end{figure}

The first identity in Proposition \ref{propo:1}\eqref{part:a} is obtained by using Pattern 1 in Figure \ref{hex_patterns:1}(a)  (or Carlitz's pattern, Figure \ref{carlitz_patterns:1}(b)), in Pascal's triangle.
For example, using Figure \ref{hex_patterns:1} (a), Pattern 1, we have that  $$\underbrace{\binom{6}{1}}_{\color{blue}{\textbf a}}+ \underbrace{\binom{6}{3}}_{\color{blue}{\textbf b}} +\underbrace{\binom{6}{5}}_{\color{blue}{\textbf c}}= 2^5.$$ 
This identity is also clear from Figure \ref{alternatingpatterns:2}.

The second identity in Proposition \ref{propo:1}\eqref{part:b} is obtained by using Pattern 2 in Figure \ref{hex_patterns:1}(a)  (or Carlitz's pattern, Figure \ref{carlitz_patterns:1}(b)), in Pascal's triangle.
For example, using in Pattern 2 we have that

 $$\underbrace{\binom{10}{1}}_{\color{blue}{\textbf d}}+\underbrace{\binom{8}{1}}_{\color{blue}{\textbf e}}+\underbrace{\binom{10}{3}}_{\color{blue}{\textbf f}}+\underbrace{\binom{8}{3}}_{\color{blue}{\textbf g}}+\underbrace{\binom{10}{5}}_{\color{blue}{\textbf h}}+\underbrace{\binom{8}{5}}_{\color{blue}{\textbf i}}+\underbrace{\binom{10}{7}}_{\color{blue}{\textbf j}}+\underbrace{\binom{8}{7}}_{\color{blue}{\textbf k}}+\underbrace{\binom{10}{9}}_{\color{blue}{\textbf l}}=5(2^7) .$$
 
 Therefore, in general, we obtain the following identities.
\begin{proposition}\label{propo:1}
If $n=2k$ for $k>0$, then the following hold
\begin{enumerate}[(a)]
\item \label{part:a} $\displaystyle \sum_{i=1}^{k}\binom{2k}{2i-1}=2^{2k-1}$,
\item \label{part:b} $\displaystyle \sum_{i=1}^{k}\binom{2k}{2i-1}+\sum_{i=1}^{k-1}\binom{2k-2}{2i-1}=5(2^{2k-3})$. 
\end{enumerate}
\end{proposition}

\begin{figure}[htbp]
\begin{center}
\includegraphics[scale=0.3]{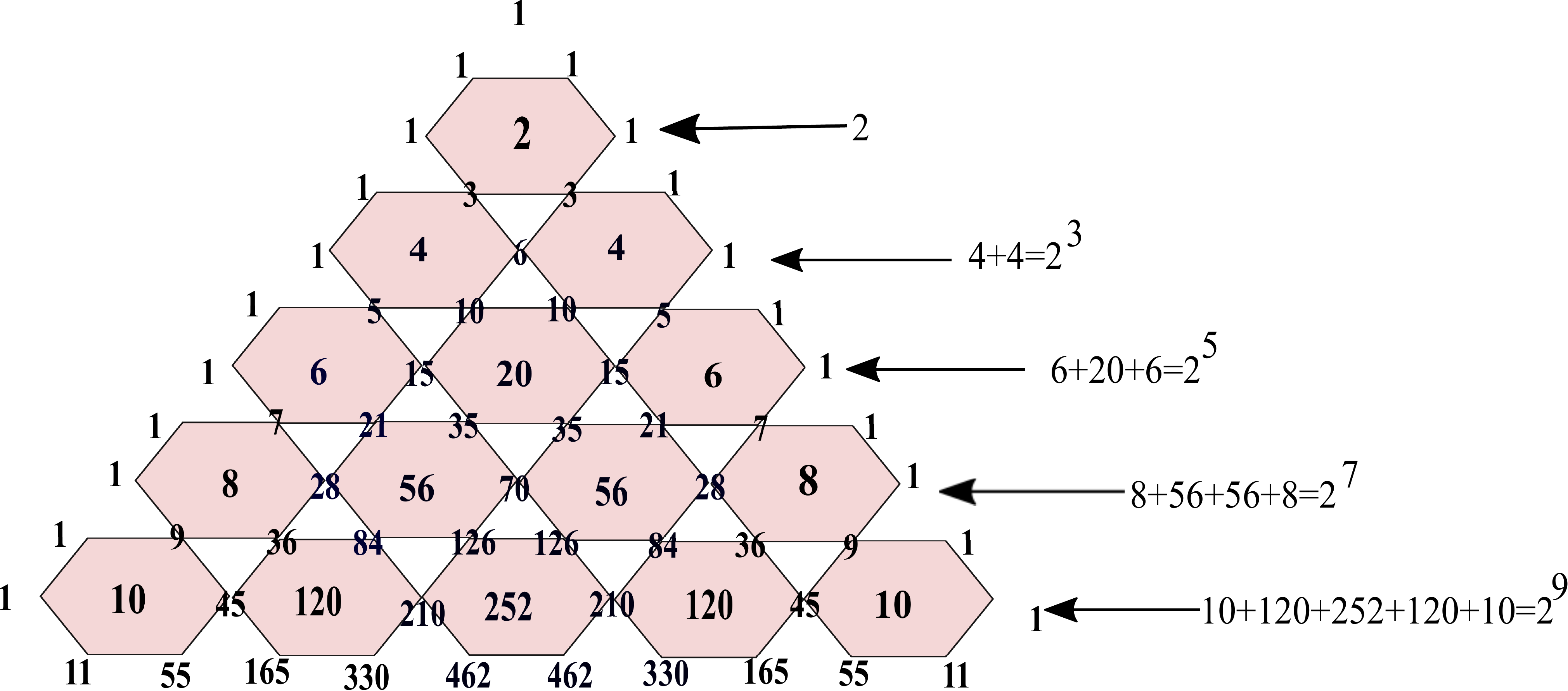}
\caption{ Power of 2 in Pascal triangle.}
\end{center}
\label{alternatingpatterns:2}
\end{figure}

The identity in Proposition \ref{propo:3} is obtained by using Pattern 3 in Figure \ref{hex_patterns:1}(b). For example, using this pattern we have that
$$\underbrace{\binom{12}{1}}_{\color{blue}{\textbf s}}+\underbrace{\binom{12}{3}}_{\color{blue}{\textbf t}}+\underbrace{\binom{10}{3}}_{\color{blue}{\textbf u}}+\underbrace{\binom{12}{5}}_{\color{blue}{\textbf v}}+\underbrace{\binom{12}{7}}_{\color{blue}{\textbf w}}+\underbrace{\binom{10}{7}}_{\color{blue}{\textbf x}}+\underbrace{\binom{12}{9}}_{\color{blue}{\textbf y}}+\underbrace{\binom{12}{11}}_{\color{blue}{\textbf z}}=2^4[9(2^4)-1]=2288.$$ 

\begin{proposition}\label{propo:3}
If $n\ge 2$ is any positive integer, then
\[\sum_{i=1}^{2n}\binom{4n}{2i-1}+\sum_{i=1}^{n-1}\binom{4n-2}{4i-1}=
\begin{cases}
2^{2n-2}[9 (2^{2n-2})+1], & \text{ if  $n$  is even};\\
2^{2n-2}[9 (2^{2n-2})-1], & \text{ if  $n$  is odd}.
\end{cases}\]
\end{proposition}

The identity in Theorem \ref{theorem:4} below is not present in existing literature and is  obtained by using the pattern in Figure \ref{carlitz_patterns:1}(d)  
(or Pattern 4 in Figure \ref{hex_patterns:1}(b)).  For example, uisng Figure \ref{hex_patterns:1}, Pattern 4 we have that

$$\underbrace{\binom{6}{1}}_{\color{blue}{ \textbf m}}+ \underbrace{\binom{4}{1}}_{\color{blue}{ \textbf n}} +\underbrace{\binom{4}{3}}_{\color{blue}{ \textbf o}}+\underbrace{\binom{6}{5}}_ {\color{blue}{ \textbf p}}= 20.$$
In general, this sum is congruent to $20\bmod 42$.

\begin{theorem}\label{theorem:4} If $n\ge 1$, then
$$\sum_{i=0}^{n-1}\left[\binom{6n}{6i+1}+\binom{6n}{6i+5}+\binom{6n-2}{6i+1}+\binom{6n-2}{6i+3}\right]\equiv \binom{6}{1}+ \binom{6}{5} +\binom{4}{1}+\binom{4}{3}= 20 \bmod\, 42.$$ 
\end{theorem}

Finally, the Carlitz's pattern see in Figure \ref{carlitz_patterns:1}(e) gives rise to the integer sequence 
$$2, 4, 4, 6, 20, 6, 8, 56, 56, 8, 10, 120, 252, \ldots=\binom{2n}{2m+1}, \text{ for } n>1 \text{ and }0\le m\le n-1.$$
This sequence is related to the integer sequence  {A091044} in \cite{sloane} with terms $1, 2, 2, 3, 10, 3, 4, 28, 28, \ldots$. 
In fact, the sequence  {A091044}, is one half of the odd-numbered entries of even-numbered rows of the Pascal triangle.

\section{Carlitz's patterns in the Hosoya triangle}\label{hosoya:triangle}

The Hosoya triangle \cite{hosoya}, is a triangular array where the entries are products of Fibonacci numbers.
The geometry of the Hosoya triangle presents an interesting method to explore properties of Fibonacci numbers.
Several properties and identities have already been discovered and published. Some articles on this topic are  by Hosoya,  Koshy, and Fl\'orez \emph{et al.} \cite{Blair,florezHiguitaJunesGCD, florezjunes, hosoya, koshy}.

\begin{table} [!h] \small
\begin{center} \addtolength{\tabcolsep}{-3pt} \scalebox{.9}{%
\begin{tabular}{cccccccccccccccccccccccccccc}
&&&&&&&&&&&&&                                                  1                                             		  &&&&&&&&&&&&&&\\
&&&&&&&&&&&&                                   	       1    &&    1                                        		    &&&&&&&&&&&&&\\
&&&&&&&&&&&                                		   2    &&     1    &&     2                                     	       &&&&&&&&&&&&\\
&&&&&&&&&&                                    3    &&     2    &&     2     &&   3                                 	            &&&&&&&&&&\\
&&&&&&&&&                         	    5    &&      3    &&    4     &&   3    && 5                                           &&&&&&&&&\\
&&&&&&&&                           8    &&      5    &&    6    && 6     &&   5     &&  8                                         &&&&&&&&\\
&&&&&&&               	    13   &&     8    && 10  &&   9     &&   10    &&   8       &&   13            		    &&&&&&&\\
&&&&&&                     21  &&  13  &&  16 && 15 && 15 && 16 && 13 && 21                                               &&&&&&\\
\end{tabular}

\begin{tabular}{cccccccccccccccccccccccccccc}
&&&&&&&&&&&& &                                                   1                                             		  &&&&&&&&&&&&&&\\
&&&&&&&&&&&&                                  	       1    &&   2                                        		   &&&&&&&&&&&&&\\
&&&&&&&&&&&                               		   1    &&     3    &&     2                                     	       &&&&&&&&&&&\\
&&&&&&&&& &                                   1   &&     4    &&     5     &&   2                                	           & &&&&&&&&&\\
&&&&&&&& &                        	    1    &&      5    &&    9    &&   7    && 2                                           &&&&&&&&&\\
&&&&&&&&                          1    &&     6    &&    14   && 16     &&  9     &&  2                                      & &&&&&&&\\
&&&&&&&               	    1   &&     7    && 20  &&   30    &&   25    &&  11       &&   2            		           & &&&&&&\\
&&&&&&                    1 && 8  &&  27 && 50 && 55 &&  36      &&      13          &&       2                                               &&&&&&\\
\end{tabular}}
\end{center}
\caption{(a) Hosoya triangle \hspace{4.0cm} (b) Lucas triangle.} \label{tabla1}
\end{table}

We represent each entry of the Hosoya triangle as $H_{r,k}= F_kF_{r-k+1}$ for positive integers $r$ and $k$ with $r \ge k\ge 1$,
(see \cite[p.~188]{koshy}).

Recently Blair  et al. \cite{Blair1}, published the article titled ``Honeycombs in the Hosoya Triangle". In this article the authors look into the patterns in Figure \ref{carlitz_patterns:1} (a) and (b) (or Patterns 1 and 2 in Figure \ref{hex_patterns:1}) inside the Hosoya triangle and this leads to discovery of some new and geometric interpretation of some known identities. In addition, we obtain a sequence of integers by 
following the pattern seen in Figure \ref{carlitz_patterns:1}(e). The examples and propositions that follow appeared in \cite{Blair1}.

For example, using Pattern 1 of Figure  \ref{hex_patterns:1}(a) we have that

$$\underbrace{H_{7,2}}_ {\color{blue}{ \textbf a}}+ \underbrace{H_{7,4}}_{\color{blue}{ \textbf  b}}+\underbrace{H_{7,6}}_{\color{blue}{ \textbf  c}}=F_2F_6+F_4F_4+F_6F_4= 25.$$

In general this gives rise to the integer sequence  {A001871} in OEIS (\cite{sloane}). Each term of this sequence is described as the convolution of the even-indexed Fibonacci numbers with itself.
This generalizes to the following identity which is a special case of a result by Fl\'{o}rez et al. in \cite{Czabarka}.

\begin{proposition}[\cite{Blair1}] \label{prop:6}
If $n>0$ is odd, then 
\[ \sum_{k=1}^{(n-1)/2}H_{(n,2k)}=\sum_{k=1}^{(n-1)/2}F_{2k}F_{n-2k+1}=\left[(2n-6)F_{n-3}+(3n-5)F_{n-2}+[(n-1)/2] \,F_{n-2}\right]/5.\]
\end{proposition}

Next, we consider Pattern 2 in Figure \ref{hex_patterns:1}(a) embedded in the Hosoya triangle. 

As an example, using Pattern 2 we have that

        $$ \underbrace{H_{11,2}}_{\color{blue}{ \textbf d}}+\underbrace{H_{9,2}}_ {\color{blue}{ \textbf e}}+\underbrace{H_{11,4}}_{\color{blue}{ \textbf  f}}+\underbrace{H_{9,4}}_{\color{blue}{ \textbf g}}+\underbrace{H_{11,6}}_{\color{blue}{ \textbf h}}+\underbrace{H_{9,6}}_{\color{blue}{ \textbf i}}+\underbrace{H_{11,8}}_{\color{blue}{ \textbf j}}+\underbrace{H_{9,8}}_{\color{blue}{ \textbf k}}+\underbrace{H_{11,10}}_{\color{blue}{ \textbf l}}$$
equals the sum $2(F_2F_{10}+F_4F_8)+F_6F_6+2(F_2F_8+F_4F_6)=390$.

Here we observe that the sums of Hosoya triangle entries $H_{n,k}$ following this pattern  ---with $n$ an odd integer--- gives rise to the integer sequence  {A197649} in OEIS (\cite{sloane}). In general we have the following identity.

\begin{proposition}[\cite{Blair1}]\label{prop:7}
If $n>0$ is odd, then 
\[\sum_{k=1}^{(n-1)/2}H_{n,2k}+\sum_{k=1}^{(n+1)/2}H_{n+2,2k}=\sum_{k=1}^{(n-1)/2}F_{2k}F_{n-2k+1}+\sum_{k=1}^{(n+1)/2}F_{2k}F_{n-2k+3}=[(n+1)F_{n+2}-2F_{n+1}]/2.\]
\end{proposition}

Finally, the pattern seen in Figure \ref{carlitz_patterns:1} (e) gives rise to the integer sequence $1, 3, 3, 8, 9, 8, 21, 24,\ldots$, which is identified as sequence  {A141678} in OEIS (\cite{sloane}).
This sequence is defined as symmetrical triangle of coefficients based on invert transform of the sequence  {A001906}. Note that the sequence  {A001906} is the bisection of the Fibonacci sequence
given by the recursive sequence $a(n)=3a_{n-1}-a_{n-2}$ for $n\ge 3$, with initial conditions $a_1=0$ and $a_1=1$. 

\section{Carlitz's patterns in other triangles}\label{other:triangle}

Carlitz's patterns can be used in other triangles like the Lucas triangle, Josef's triangle, and Liebniz's harmonic triangle to explore and obtain identities related to several different integer sequences.
The Lucas triangle (Table \ref{tabla1}(b)) is constructed using coefficients of polynomials of the form $f_n(x,y)=(x+y)^{(n-1)(x+2y)}$. If this triangle is left-justified, then its rising diagonals give rise to Lucas numbers.
On the other hand, Josef's triangles can be constructed like the Hosoya triangle by starting with two outermost columns of Lucas and Fibonacci number sequences (Table \ref{tabla2}(a)). Leibniz's triangle (Table \ref{tabla2}(b)) can be obtained by using the following recursive relation

\[L(n, 1) = 1/n; L(n, k) = L(n-1, k-1)-L(n, k-1), \text{for } k>1.\]
All three of these triangles are described in detail in  \cite{koshy2}.

For all three triangles Patterns 1 and 2 in Figure \ref{hex_patterns:1} yields some known identities or integer sequences found in \cite{sloane}.

In particular for the Lucas triangle, Pattern 1 (sum of every other entry in the odd-numbered rows), yields the integer sequence $3,12,48,192, \dots$ 
which is given by the formula $a_n=3(4^{n-1})$ for $n\ge 1$. This is also present in  OEIS \cite{sloane} as sequence  {A002001}.  
Pattern 2 in the Lucas triangle yields the sequence $3,15,60,240,\ldots$.
 
 Similarly, for Josef's triangle, Pattern 1, yields the integer sequence $2, 10, 40,1 42, 470\ldots$, which is related to the sequence of self-convoluted  
 Fibonacci numbers (  {A001629} in OEIS). In fact, one-half of each term of this sequence, namely $1, 5, 20, 71, 235,\ldots$, is the sequence of 
 even-indexed terms of  {A001629}. The sequence $1, 5, 20, 71, 235,\ldots$, is identified in OEIS (\cite{sloane} ) as  {A054444}.
 
 In Josef's triangle, using Pattern 2 we get the sequence $12,50,182,612,\ldots$ where each term is the sum of consecutive terms of the sequence  
 $2, 10, 40, 142, 470, \ldots$, obtained from Pattern 1 above.
 
 Finally, for the Leibniz's harmonic denominator triangle, Pattern 1 gives us the following sequence $$6, 40, 224,1 152, 5632,\ldots.$$ This sequence can be given by the recursive relation $D_n=8D_{n-1}-16D_{n-2}$ for $n>3$ with initial conditions $D_1=1, D_2=6$, and $D_3=40$.
 This sequence is identified in  \cite{sloane} as sequence  {A229580}. In the same triangle, Pattern 2 gives us the sequence $46,264,1376,7784,\ldots$. Each term of this sequence is the sum of consecutive terms of the sequence  $6,40,224,1152, 5632,\ldots.$ obtained from Pattern 1.
 
\begin{table} [!h] \small
\begin{center} \addtolength{\tabcolsep}{-3pt} \scalebox{.8}{%
\begin{tabular}{cccccccccccccccccccccccccccc}
&&&&&&&&&&&&&                                                  2                                             		  &&&&&&&&&&&&&&\\
&&&&&&&&&&&&                                   	       1    &&    1                                        		    &&&&&&&&&&&&&\\
&&&&&&&&&&&                                		   3    &&     2    &&     3                                    	       &&&&&&&&&&&&\\
&&&&&&&&&&                                    4   &&     3    &&     3     &&   4                                 	            &&&&&&&&&&\\
&&&&&&&&&                         	    7    &&      5    &&    6     &&   5    && 7                                           &&&&&&&&&\\
&&&&&&&&                           11   &&      8   &&   9    && 9     &&   8     && 11                                       &&&&&&&&\\
&&&&&&&               	    18   &&     13    && 15  &&  14     &&   15    &&  13       &&  18            		    &&&&&&&\\
&&&&&&                     29  &&  21  &&  24 && 23 && 23 && 24 &&   21 &&      29                                              &&&&&&\\
\end{tabular}
\begin{tabular}{cccccccccccccccccccccccccccc}
&&&&&&&&&&&& &                                                   1                                             		  &&&&&&&&&&&&&&\\
&&&&&&&&&&&&                                  	      $\frac{1}{2}$   &&   $\frac{1}{2}$                                 &&&&&&&&&&&&&\\
&&&&&&&&&&&                               		  $\frac{1}{3}$    &&    $\frac{1}{6}$   &&    $\frac{1}{3}$             &&&&&&&&&&&\\
&&&&&&&&& &                 $\frac{1}{4}$   &&     $\frac{1}{12}$    &&     $\frac{1}{12}$ &&   $\frac{1}{4}$          & &&&&&&&&&\\
&&&&&&&& &  $\frac{1}{5}$   &&      $\frac{1}{20}$    &&    $\frac{1}{30}$    &&   $\frac{1}{20}$   && $\frac{1}{5}$                                          &&&&&&&&&\\
&&&&&&&&                        $\frac{1}{6}$     &&     $\frac{1}{30}$   &&   $\frac{1}{60}$   &&  $\frac{1}{60}$    &&  $\frac{1}{30}$    &&  $\frac{1}{6}$                                      & &&&&&&&\\
&&&&&&&               	    $\frac{1}{7}$   &&     $\frac{1}{42}$    && $\frac{1}{105}$  &&  $\frac{1}{140}$    &&   $\frac{1}{105}$    &&  $\frac{1}{42}$      &&  $\frac{1}{7}$        		           & &&&&&&\\
&&&&&&                    $\frac{1}{8}$ && $\frac{1}{56}$  &&  $\frac{1}{168}$ && $\frac{1}{280}$ && $\frac{1}{280}$ &&  $\frac{1}{168}$     &&      $\frac{1}{56}$          &&       $\frac{1}{8}$                                               &&&&&&\\
\end{tabular}}
\end{center}
\caption{ (a) Josef's triangle \hspace{3.5cm} (b) Leibniz harmonic triangle} \label{tabla2}
\end{table}

\section{Appendix --Proofs}
In this section we present some proofs of the identities in case the reader is interested in the formal proofs of the results presented earlier.

 First we present some well-known identities involving binomial coefficients which will be used to prove the propositions in this section.  
 Lemma \ref{lemma:1} Parts \eqref{lemma:parta} and \eqref{lemma:partb} and their proofs can be found in any combinatorics book, for example 
 \cite{gould1, spivey}.   Lemma \ref{lemma:1}, part \eqref{lemma:partc} is a special case of an identity  
 found in \cite{gould1}.  The proofs of Proposition \ref{propo:1} Parts \eqref{part:a} and \eqref{part:b} follow directly from Lemma \ref{lemma:1} 
 Parts \eqref{lemma:parta} and \eqref{lemma:partb}.
 
\begin{lemma}\label{lemma:1}
If $n,k$ are positive integers, then the following hold
\begin{enumerate}[(a)]
\item $\displaystyle\binom{n+1}{k}=\binom{n}{k}+\binom{n}{k-1}$.\label{lemma:parta}
\item $\displaystyle\sum_{k=0}^n \binom{n}{k} =2^n$.\label{lemma:partb}
\item $\displaystyle\sum_{k=0}^{2n}(-1)^k\binom{4n}{2k}=(-1)^n2^{2n}$. \label{lemma:partc}
\end{enumerate}
\end{lemma}

\begin{proof}[Proof of Proposition \ref{propo:3}]
We first note that the proof of $\displaystyle\sum_{i=1}^{2n}\binom{4n}{2i-1}=2^{4n-1}$ is well-known (see \cite{gould1, spivey}). 
Therefore, we only  prove that
\begin{equation}\label{eqn:1}
\sum_{i=1}^{n-1}\binom{4n-2}{4i-1}=
\begin{cases}
2^{2n-2}[2^{2n-2}+1] & \text{ if } n \text{ is even};\\
2^{2n-2}[2^{2n-2}-1] & \text{ if } n \text{ is odd}.
\end{cases}
\end{equation}

Observe that Lemma  \ref{lemma:1} part (a) implies that \[\displaystyle\sum_{i=1}^{n-1}\binom{4n-2}{4i-1}=\sum_{i=1}^{n-1}\binom{4n-3}{4i-1}+\binom{4n-3}{4i-2}.\]

Using the same identity again and regrouping the terms in the sum, we have that

\[\sum_{i=1}^{n-1}\binom{4n-2}{4i-1}=\sum_{i=0}^{4n-4}\binom{4n-4}{i}-\sum_{i=0}^{n-1}\binom{4n-4}{4i}+\sum_{i=1}^{n-1}\binom{4n-4}{4i-2}.\]

Further simplification gives us, $\displaystyle\sum_{i=1}^{n-1}\binom{4n-2}{4i-1}=\sum_{i=0}^{4n-4}\binom{4n-4}{i}-\sum_{i=0}^{2n-2}(-1)^i\binom{4n-4}{2i}$. Finally, applying Lemma \ref{lemma:1} Parts \eqref{lemma:parta} and \eqref{lemma:partb}, we obtain

\[\sum_{i=1}^{n-1}\binom{4n-2}{4i-1}=
\begin{cases}
2^{2n-2}[2^{2n-2}+1] & \text{ if } n-1 \text{ is odd};\\
2^{2n-2}[2^{2n-2}-1] & \text{ if } n-1\text{ is even,}
\end{cases}
\]
which is equivalent to \eqref{eqn:1}.
This completes the proof.
\end{proof}

\begin{proof}[Proof of Theorem \ref{theorem:4}]
Note that, 
$$\sum_{i=0}^{n-1}\left[\binom{6n}{6i+1}+\binom{6n}{6i+5}+\binom{6n-2}{6i+1}+\binom{6n-2}{6i+3}\right]$$
is congruent with  $0 \bmod 2$, is congruent with  $2 \bmod 3$, \text{ and }  is congruent with  $6 \bmod 7.$
By the Chinese Remainder Theorem, we can easily solve this system of equation. This proves the theorem.
\end{proof}

\section{Acknowledgement}
	
The second and third authors were partially supported by The Citadel Foundation.

\bigskip
\hrule
\bigskip
	
\noindent 2010 {\it Mathematics Subject Classification}:
Primary 11B39; Secondary 11B83.
	
\noindent \emph{Keywords: }
Pascal triangle, Hosoya triangle,  Carlitz's patterns.	
\end{document}